\documentclass[10pt]{amsart}
\usepackage{amsmath,amscd}
\usepackage{amsbsy}
\usepackage{amssymb}
\usepackage{amscd,amsthm}
\usepackage[all,cmtip]{xy}
\usepackage[russian,english]{babel}
\usepackage[utf8]{inputenc}
\usepackage{xcolor}
\usepackage{dsfont}

\usepackage{hyperref}

\renewcommand{\epsilon}{\varepsilon}

\renewcommand{\ell}{x}

\newtheorem{thm}{Theorem}\numberwithin{thm}{section}
\newtheorem{lem}[thm]{Lemma}

\newtheorem*{con2}{Conjecture}

\newtheorem*{rema2}{Remark}

\begin{document}
	\begin{center}
		\huge{A comment on the equation $n!!=a_1!!\cdots a_t!!$}\\[1cm] 
	\end{center}
	\begin{center}
		\large{Sa$\mathrm{\check{s}}$a Novakovi$\mathrm{\acute{c}}$}\\[0,5cm]
		{\small April 2026}\\[0,2cm]
	\end{center}
		{\small \textbf{Abstract}. 
			We study the equation $a_1!!\cdots a_t!!=n!!$ and show that in certain special cases the explicit abc conjecture implies that it has only finitely many nontrivial solutions.}
		\begin{center}
			\tableofcontents
		\end{center}
		\section{Introduction}
		\noindent
		It is a prominent and long-standing unsolved problem whether the equation
		$$
		\prod_{i=1}^ta_i!=n! \quad n>a_1\geq a_2\geq \cdots \geq a_t\geq 2
		$$
		has only finitely many solutions. For a positive integer $n$ we write $P(n)$ for the largest prime factor of $n$. Erd\H{o}s and Graham \cite{EG} made the observation that if
		$$
		\frac{P(n(n+1))}{\mathrm{log}(n)}\longrightarrow \infty \quad \textnormal{as} \quad n\rightarrow \infty
		$$
		then the equation would have only finitely many solutions. A trivial solution has $a_1=n-1$. Thus, $n=a_2!\cdots a_t!$ and consequently, there are infinitely many trivial solutions. The Sur$\mathrm{\acute{a}}$nyi--Hickerson conjecture predicts that $16!=14!5!2!$ is the largest nontrivial solution and Nair and Shorey \cite{NS} showed that under the Baker's explicit abc conjecture, the equation has indeed only the nontrivial solutions 
		$$
		7!3!3!2!=9!, 7!6!=10!, 7!5!3!=10!, 14!5!2!=16!, 15!2!^4=16!
		$$ 
		The classical abc conjecture barely implies $P(n(n+1))\geq (1+o(1))\mathrm{log}(n)$ as $n\rightarrow \infty$, which is weaker the the estimate of Erd\H{o}s and Graham. However, Luca \cite{LUC} showed that the classical abc conjecture implies that the above equation has only finitely many nontrivial solutions. This result was sharpend by Luca, Saradha and Shorey \cite{LUC2} under Baker's explicit abc conjecture and explicit bounds for $a_2$ were obtained. In this present note, we focus on the equation
		\begin{equation}
			\prod_{i=1}^ta_i!!=n!! \quad n>a_1\geq \cdots a_t\geq 3,\quad  t>1.
		\end{equation}
		\noindent
		Here $m!!$ denotes the double factorial of a positive integer $m$. It is defined as follows: if $m=2l-1$ is odd, $m!!=1\cdot 3\cdots 2l-1$ and if $m=2l$ is even, $m!!=2\cdot 4\cdots 2l$. Nair and Shorey \cite{NSH} studied equation (1) in the case where all $a_i$ and $n$ are odd. In this case, trivial solutions are given by $n-a_1=2$ and there are infinitely many of them. They proved that Baker's explicit abc conjecture implies that there are only finitely many nontrivial solutions. Moreover, they gave explicit bounds for $n-a_1$ and $a_1$. In the present comment, we consider the case where $a_2,...,a_t$, $t\geq 2$ are all even. Notice that $a_1$ is allowd to be even or odd. Now we have to say something about \emph{trivial} solutions in these cases. So if all $a_i=2A_i$ are even, $n=2N$ must be even as well and $n-a_1=2$ gives the equation $a_2!!\cdots a_t!!=2^{A_2}A_2!\cdots 2^{A_t}A_t!!=2N$ and we see that there are infinitely many such solutions, simply by setting $n=a_2!!\cdots a_t!!$. Now if $a_1$ is odd and $a_2,...,a_t$ are even, there is no trivial solution in the sense that $n-a_1=l$ for a fixed positiv integer $l$. Let us see why. So let $l$ be a positive but fixed integer and consider all $n$ and $a_1$ such that $n-a_1=l$. Dusart \cite{DU} in his thesis proved that there is a prime $p\in (y/(1+2\mathrm{log}^2(y)),y)$ for $y\geq 3275$. So we set $y=a_1$. Hence there is a $n_0$ such that for $n\geq n_0$, we get
		$$
		p>\frac{a_1}{(1+\frac{1}{2\mathrm{log}^2(a_1)})}=\frac{n-l}{(1+\frac{1}{2\mathrm{log}^2(n-l)})}>\frac{n}{2}.
		$$
		Notice that $n$ is even and therefore $n!!=2^{n/2}(n/2)!$. So for sufficiently large $n$, there is a prime $p$ dividing $a_1!!$ but not $n!!$. And this implies that the equation
		$$
		a_1!!\cdots a_t!!=n!!
		$$
	where $a_1$ is odd and $a_2,...,a_t,n$ are even has no solution for $n\geq n_0$ if $l $ is fixed. On the next page we will see that there also exist infinitely many \emph{trivial} solutions when $a_1$ is odd, but not if $n-a_1$ is fixed. In the case $a_1$ even, solutions of the type $n-a_1=2$ are called \emph{trivial}.
	And of course, we are interested in nontrivial solutions.  
	Since we apply the (explicit) abc conjecture in proving Theorems 1.1 and 1.2, we recall its statement. For a non-zero integer $a$, let $N(a)$ be the \emph{algebraic radical}, namely $N(a)=\underset{p\mid a}{\prod}{p}$.
		\begin{con2}[classical abc conjecture]
			For any $\epsilon >0$ there is a constant $K(\epsilon)$ depending only on $\epsilon$ such that whenever $a,b$ and $c$  are three coprime and non-zero integers with $a+b=c$, then 
			\begin{eqnarray*}
				c<K(\epsilon)N(abc)^{1+\epsilon}
			\end{eqnarray*}
			holds.
		\end{con2} 
		\noindent
		In 1975 Baker \cite{BAK} gave an explicit version of the abc conjecture. Then Laishram and Shorey \cite{LS}, Theorem 1 showed that Baker's version implies that
		$$
		c<N(abc)^{7/4}.
		$$ 
		We shall use this inequality in place of classical abc or Baker's explicit abc conjecture in proving our main results. We simply say \emph{explicit abc conjecture}. Now we specify our problem. Denote by $r$ the number of $1\leq i\leq t$ such that the $a_i$ in equation (1) is odd. Then either $r=t$ or $0\leq r<t$. As mentioned above, the case $r=t$, and hence $n$ odd, was addressed by Nair and Shorey \cite{NSH}. If $r=t-1$ then there is an $i$ such that $a_i=2A_i$ and therefore $n=2N$ must be even as well. Our equation then looks like
		$$
		\underset{l\neq i}{\prod}a_l!!=2^{N-A_i}\frac{N!}{A_1!}.
		$$
		And since $n>a_i$ by assumption, this equation has no solution, because the left hand side is odd. So this implies $0\leq r\leq t-2$. Therefore, we may assume $t>2$ throughout the work. Note that this means that $n$ must be even. 
		In the present work we consider the special case $0\leq r\leq 1$. Our main results are Theorems 1.1 and 1.2 below. 
		\begin{thm}
			Let $r=0$. Under the explicit abc conjecture, equation (1) has finitely many nontrivial solutions.
			\end{thm}
			\noindent
			Now we consider equation (1) with $a_1$ odd and $a_2,...,a_t$ even. We write $n=2N, a_1=2A_1-1, a_2=2A_2,...,a_t=2A_t$. Then equation (1) becomes
			$$
			2^NN!=a_1!!2^{A_2}A_2!\cdots 2^{A_t}A_t!,
			$$
			or
			$$
			2^{A_1}A_1!2^NN!=(a_1+1)!2^{A_2}A_2!\cdots 2^{A_t}A_t!.
			$$
			We set $x_1=A_t+1, l_1=N-A_t$ and $x_2=A_2+1, l_2=A_1-A_2$. Denote by $\Delta(x,l)=x(x+1)\cdots (x+l-1)$. We can rewrite the above equation as 
			\begin{equation}
				2^{l_1}\Delta(x_1,l_1)2^{l_2}\Delta(x_2,l_2)=(a_1+1)!\cdots 2^{A_{t-1}}A_{t-1}!.
			\end{equation}
			Notice that $x_1+l_1=N+1>A_1+1=x_2+l_2$ and $l_1\geq l_2$, by construction. We now can explain what \emph{trivial} solution means when $a_1$ is odd. Let $l_1=N-A_t=1$. Then $l_2=A_1-A_2=1$ as well, because we assumed $n>a_1>a_2\geq a_3\geq \cdots \geq a_t\geq 3$. We consider the equation
			$$
				2^{l_1}\Delta(x_1,l_1)2^{l_2}\Delta(x_2,l_2)=(a_1+1)!\cdots 2^{A_{t-1}}A_{t-1}!.
			$$
			So if $l_1=1=l_2$, the equation becomes
			$$
			2N2A_1=(a_1+1)!a_3!!\cdots a_{t-1}!!=(2A_1)!a_3!!\cdots a_{t-1}!!.
			$$
			But this gives 
			$$
			n=2N=(2A_1-1)!a_3!!\cdots a_{t-1}!!=a_1!a_3!!\cdots a_{t-1}!!.
			$$
			Therefore, there are infinitely many solutions when $a_1$ is odd, simply by setting $n=a_1!a_3!!\cdots a_{t-1}!!$, $N-1=A_t$ and $A_1-1=A_2$. We will call such solutions \emph{trivial}.
			\begin{thm}
				Let $a_1$ be odd and $a_2,...,a_t$ be even. Then the following holds:
				\begin{itemize}
					\item [(i)] if $\Delta(x_1,l_1)$ containes no prime, then the explicit abc conjecture implies that there are finitely many nontrivial solutions,
					\item[(ii)] if $\Delta(x_1,l_1)$ containes a prime, then there are finitely many solutions if $x_1>4l_1$. If $x_1\leq 4l_1$, then 
					$$
					0.99a_1\leq 4l_1+4.01+\frac{2\mathrm{log}(5)}{\mathrm{log}(2)}+2\frac{\mathrm{log}(l_1)}{\mathrm{log}(2)}.
					$$ 
				We see that $a_1$ is bounded whenever $l_1$ is bounded.	 
				\end{itemize}
			\end{thm}
		\begin{rema2}
			\textnormal{One could make Theorem 1.1 explicit by keeping track of all the constants and coefficients in the estimates that appear in the proof. We don't see how the methods used in the proofs of Theorems 1.1 and 1.2 could work for $2\leq r\leq t-2$ if $t\geq 4$. Or even for $r=1$ with $a_t$ odd. In the proof we use the product $\Delta(m,k)$ of consecutive integers and the main obstacle for the mentioned cases seems to be that the numbers in the product are not composite integers. So we don't know how to control the largest prime $P(\Delta(m,k))$ in such a way that we obtain the desired estimates. Similar observation was made by Nair and Shorey in \cite{NSH}.}
		\end{rema2}
	
		\section{Preliminaries}
		We state some facts from number theory that we need. As mentioned in the introduction, $P(n)$ denotes the greatest prime factor of $n$. We put $P(1)=1$. Let $\Delta(x,k)=x(x+1)\cdots (x+k-1)$. Erd\H{o}s \cite{ERD} proved that there exists a large number $\kappa$ such that
		$$
		P(\Delta(x,k))>\frac{2}{7}k\mathrm{log}(k)\quad \textnormal{for} \quad k\geq \kappa
		$$
		whenever $x,x+1,...,x+k-1$ are all composite integers. In fact, Erd\H{o}s proved the result for an unknown constant in place of $2/7$. A proof with $2/7$ can be found in \cite{LUC2}, Lemma 7. Next, we recall two well known estimates. Fix a positive real number $\nu>1$. Let $\theta(\nu)=\sum_{p\leq \nu}\mathrm{log}(p)$.
		\begin{lem}
			We have
			\begin{itemize}
				\item [(i)] $\theta(\nu)<1.00008\nu$ \textnormal{for} $\nu>1$.
				\item[(ii)] $\underset{p\leq \nu}{\sum}\frac{\mathrm{log}(p)}{p}<\mathrm{log}(\nu)$ \textnormal{for} $\nu>1$. 
			\end{itemize}
		\end{lem}
		\begin{proof}
			For a proof of (i) see \cite{DU} and \cite{DU1}. For (i) see \cite{RO}. 
		\end{proof}
		\noindent
		As mentioned in the introduction, we assume $a_2,...,a_t$ to be even. With regard to equation (1), we study two cases, 
		namely $a_1=2A_1$ is even or $a_1=2A_1-1$ is odd. Let $a_1$ be even. We set $m=A_1+1$ and $k=N-A_1$ and let $\Delta(m,k)=m(m+1)\cdots (m+k-1)$. In the case where $a_1$ is even, $n-a_1=2$ gives trivial solutions. So me may assume $k\geq 2$. 
		\begin{lem}
			Let $a_1$ be even and consider $\Delta(m,k)$ from above. Then the following holds.
			\begin{itemize}
				\item [(i)] None of the terms in $\Delta(m,k)$ is a prime,
				\item[(ii)] $a_2\cdot\mathrm{log}(a_2)-a_2\leq \mathrm{log}(a_2!!)\leq k\cdot \mathrm{log}(4m)$.
			\end{itemize}
		\end{lem}
		\begin{proof}
			By assumption, let $a_1=2A_1$. Since $a_2,...,a_t$ are even, we set $a_i=2A_i$ for $2\leq i\leq t$. In this case our equation becomes
			$$
			\prod_{i=2}^ta_i!!=2^{N-A_1}\Delta(m,k)=2^{k}\Delta(m,k).
			$$
			\noindent
			Assume $m+d=p$ is prime for some $0\leq d<k$. Then 
			$$
			A_1+1\leq m+d=p\leq A_2, 
			$$
			since $A_1+1\geq 3$ by assumption. But this contradicts $a_1\geq a_2\geq \cdots \geq a_t$, since $2(A_1+1)=a_1+2>a_2=2A_2$. We now prove (ii). Suppose $m<k$. By Bertrand's postulate, there is a prime in $((m+k-1)/2, m+k-1)$. But this interval is contained in $[m,m+k-1]$, which is a contradiction, since none of the terms in $\Delta(m,k)$ is a prime according to (i). This shows $m\geq k$. By the above equation and by Stirling approximation for $a_2!!$, we obtain
			$$
			a_2\cdot \mathrm{log}(a_2)-a_2\leq \mathrm{log}(a_2!!)\leq \mathrm{log}((2(m+k))^{k})\leq k\cdot \mathrm{log}(4m).
			$$
		\end{proof}
			\begin{lem}
			Consider the term $\Delta(x_1,l_1)$ in equation (2) and assume none of the factors of $\Delta(x_1,l_1)$ is a prime. Then
			$$
			a_1\cdot\mathrm{log}(a_1)-a_1\leq \mathrm{log}((a_1+1)!)\leq 2l_1\cdot \mathrm{log}(4x_1).
			$$
		\end{lem}
		\begin{proof}
			Suppose $x_1<l_1$. By Bertrand's postulate, there is a prime in $((x_1+l_1-1)/2, x_1+l_1-1)$. But this interval is contained in $[x_1,x_1+l_1-1]$, which is a contradiction, since none of the terms in $\Delta(x_1,l_1)$ is a prime by assumption. This shows $x_1\geq l_1$. By equation (2) and by Stirling approximation for $(a_1+1)!$, we obtain
			$$
			a_1\cdot \mathrm{log}(a_1)-a_1\leq \mathrm{log}((a_1+1)!)\leq \mathrm{log}((2(x_1+l_1))^{l_1+l_2})\leq 2l_1\cdot \mathrm{log}(4x_1).
			$$
			\end{proof}
			\noindent
			Let $T$ be the following set:
		\begin{multline*}
				T=\{(9,2), (14,2), (20,2), (24,2), (27,2), (35,2), (48,2), (49,2), (63,2), (80,2),\\
				(125,2), (224,2), (2400,2), (4374,2), (13,3), (14,3), (20,3), (24,3), (25,3), (26,3),\\
				 (48,3), (54,3), (63,3), (64,3), (98,3), (350,3), (24,4), (25,4), (32,4), (33,4), (48,4),\\
				(49,4), (63,4), (24,5), (32,5), (48,5), (29,7), (30,7)\}.
			\end{multline*}
	\begin{thm}[\cite{NS}, Theorem 2.]
				Let $x>4k$, $k\geq 2$ be such that $(x,k)\notin T$. Then
				$$
				P(\Delta(x,k))>4.42k.
				$$
			\end{thm}
		\section{Proof of Theorem 1.1}
		\noindent
		We follow the line of argument as in \cite{NSH} with some minor changes. Consider equation
		$$
		\prod_{i=2}^ta_i!!=2^{N-A_1}\Delta(m,k)=2^{k}\Delta(m,k).
		$$
		We want to prove that $a_2$ is bounded. By Lemma 2.2 (i), no term in $\Delta(m,k)$ is a prime. Observe that primes $>k$ divide at most one term of $\Delta(m,k)$. For every prime $\leq k$, we delete the term in which it appears to the maximum power. Further, all primes dividing $\Delta(m,k)$ are $\leq a_2$. Thus
		\begin{equation*}
			\begin{split}
				\prod_{i=0}^{k-1}N(m+i)\leq \left(\underset{k\leq p\leq a_2}{\prod} p\right)\left( \underset{p<k}{\prod}p^{\lfloor k/p \rfloor}\right) \\
				\leq \mathrm{exp}\left( \underset{k\leq p\leq a_2}{\sum}\mathrm{log}(p)+k\underset{p<k}{\sum}\frac{\mathrm{log}(p)}{p}\right).
			\end{split}
		\end{equation*}
		By Lemma 2.1, we get
		$$
		\prod_{i=0}^{k-1}N(m+i)\leq \mathrm{exp}(1.00008a_2+k\cdot \mathrm{log}(k)).
		$$
		Choose $m+j_1$ and $m+j_2$ such that $N(m+j_1)\leq N(m+j_2)$ are the smallest among the $N(m+i)$ for $0\leq i<k$. Then
		$$
		N(m+j_2)\leq \left(\prod_{i=0}^{k-1}N(m+i) \right)^{1/(k-1)}\leq \mathrm{exp}\left(\frac{1.00008a_2}{k-1}+\frac{k\cdot \mathrm{log}(k)}{k-1} \right).
		$$
		Now consider 
		$$
		\frac{m+j_1}{d}-\frac{m+j_2}{d}=\frac{j_1-j_2}{d},
		$$	
		where $d=\mathrm{gcd}(m+j_1,j_1-j_2)$. Applying the explicit abc conjecture yields
		$$
		\frac{m}{d}\leq \left( N(m+j_1)N(m+j_2)\cdot \frac{|j_1-j_2|}{d}\right)^{7/4}.
		$$
		Hence
		$$
		\mathrm{log}(m)\leq \frac{7}{4}\left(\frac{2.00016a_2}{k-1}+\frac{2k\cdot \mathrm{log}(k)}{k-1}+\mathrm{log}(k) \right)
		$$
		and therefore
		\begin{equation}
			k\cdot \mathrm{log}(m)\leq \frac{7}{4}\left(\frac{k\cdot 2.00016a_2}{k-1}+\frac{2k^2\cdot \mathrm{log}(k)}{k-1}+k\cdot \mathrm{log}(k) \right).
		\end{equation}
		We explained in Section 2 right before Lemma 2.2 that me may assume $k\geq 2 $. This gives $k/(k-1)\leq 2$. Furthermore, because $P(\Delta(m,k))>\frac{2}{7}k\cdot \mathrm{log}(k)$ for $k\geq \kappa$ , we conclude 
		\begin{equation}
			a_2>\frac{2}{7}k\cdot \mathrm{log}(k) \quad \textnormal{for} \quad k\geq \kappa.
		\end{equation}
		First, we consider $k\geq \kappa$. From (3) and (4), we conclude that there exists a positive constant $c_1$ such that
		\begin{equation}
			k\cdot \mathrm{log}(m)\leq c_1\cdot a_2.
		\end{equation}
		Now use inequality (5) and Lemma 2.2 (ii) to obtain
		$$
		a_2\cdot \mathrm{log}(a_2)-a_2\leq k\cdot\mathrm{log}(4m)\leq c_2a_2, 
		$$
		where $c_2$ is a positive constant. Resolving to $\mathrm{log}(a_2)$ yields
		$$
		\mathrm{log}(a_2)\leq c_3
		$$
		for some constant $c_3>0$. Now consider the case $k<\kappa$. We may assume $a_2>c_4\kappa\cdot \mathrm{log}(\kappa)$ for a suitable constant $c_4>0$, since otherwise the result follows. Then (3) yields
		$$
		k\cdot \mathrm{log}(m)\leq c_5a_2+c_6k\cdot \mathrm{log}(k),
		$$
		where $c_5,c_6$ are positive constants. Again, we use Lemma 2.2 (ii) to get
		$$
		a_2\cdot \mathrm{log}(a_2)\leq c_7a_2+c_8k\cdot\mathrm{log}(k)
		$$
		and hence
		$$
		\mathrm{log}(a_2)\leq c_7+\frac{c_8}{c_4}\frac{k\cdot \mathrm{log}(k)}{\kappa\cdot \mathrm{log}(\kappa)}\leq c_7+\frac{c_8}{c_4}.
		$$
		This shows that $a_2$ is bounded when $a_1$ is even.
	
	\section{Proof of Theorem 1.2}
	\noindent
			We first prove (i). Consider equation (2). As in the proof of Theorem 1.1, we can show that 
			$$
			\prod_{i=0}^{l_1-1}N(x_1+i)\leq \mathrm{exp}(1.00008(a_1+1)+l_1\cdot \mathrm{log}(l_1)).
			$$
				Choose $x_1+j_1$ and $x_1+j_2$ such that $N(x_1+j_1)\leq N(x_1+j_2)$ are the smallest among the $N(x_1+i)$ for $0\leq i<l_1$. Then
			$$
			N(x_1+j_2)\leq \left(\prod_{i=0}^{l_1-1}N(x_1+i) \right)^{1/(l_1-1)}\leq \mathrm{exp}\left(\frac{1.00008(a_1+1)}{l_1-1}+\frac{l_1\cdot \mathrm{log}(l_1)}{l_1-1} \right).
			$$
			Now consider 
			$$
			\frac{x_1+j_1}{d}-\frac{x_1+j_2}{d}=\frac{j_1-j_2}{d},
			$$	
			where $d=\mathrm{gcd}(x_1+j_1,j_1-j_2)$. Applying the explicit abc conjecture yields
			$$
			\frac{x_1}{d}\leq \left( N(x_1+j_1)N(x_1+j_2)\cdot \frac{|j_1-j_2|}{d}\right)^{7/4}.
			$$
			Hence
			$$
			\mathrm{log}(x_1)\leq \frac{7}{4}\left(\frac{2.00016(a_1+1)}{l_1-1}+\frac{2l_1\cdot \mathrm{log}(l_1)}{l_1-1}+\mathrm{log}(l_1) \right)
			$$
			and therefore
			\begin{equation}
				l_1\cdot \mathrm{log}(x_1)\leq \frac{7}{4}\left(\frac{l_1\cdot 2.00016(a_1+1)}{l_1-1}+\frac{2l_1^2\cdot \mathrm{log}(l_1)}{l_1-1}+l_1\cdot \mathrm{log}(l_1) \right).
			\end{equation}
			Since we are interested in nontrivial solutions, me may assume $l_1\geq 2 $. This gives $l_1/(l_1-1)\leq 2$. Furthermore, because $P(\Delta(x_1,l_1))>\frac{2}{7}l_1\cdot \mathrm{log}(l_1)$ for $l_1\geq \kappa$ , we conclude 
			\begin{equation}
				a_1+1>\frac{2}{7}l_1\cdot \mathrm{log}(l_1) \quad \textnormal{for} \quad l_1\geq \kappa.
			\end{equation}
			First, we consider the case $l_1\geq \kappa$. From (6) and (7), we conclude that there exists a positive constant $c_9$ such that
			\begin{equation}
				l_1\cdot \mathrm{log}(x_1)\leq c_9\cdot (a_1+1).
			\end{equation}
			Now use inequality (8) and Lemma 2.3 to obtain
			$$
			a_1\cdot \mathrm{log}(a_1)-a_1\leq 2l_1\cdot\mathrm{log}(4x_1)\leq c_{10}(a_1+1), 
			$$
			where $c_{10}$ is a positive constant. Resolving to $\mathrm{log}(a_1)$ yields
			$$
			\mathrm{log}(a_2)\leq c_{11}
			$$
			for some constant $c_{11}>0$. Now consider the case $l_1<\kappa$. We may assume $a_1>c_{12}\kappa\cdot \mathrm{log}(\kappa)$ for a suitable constant $c_{12}>0$, since otherwise the result follows. Then (6) yields
			$$
			l_1\cdot \mathrm{log}(x_1)\leq c_{13}(a_1+1)+c_{14}l_1\cdot \mathrm{log}(l_1),
			$$
			where $c_{13},c_{14}$ are positive constants. Again, we use Lemma 2.3 to get
			$$
			a_1\cdot \mathrm{log}(a_1)\leq c_{15}(a_1+1)+c_{16}l_1\cdot\mathrm{log}(l_1)
			$$
			and hence
			$$
			\mathrm{log}(a_1)\leq c_{17}+\frac{c_{16}}{c_{12}}\frac{l_1\cdot \mathrm{log}(l_1)}{\kappa\cdot \mathrm{log}(\kappa)}\leq c_{17}+\frac{c_{16}}{c_{12}}.
			$$
			This shows that $a_1$ is bounded when $a_1$ is odd.\\
			\noindent
			We now show (ii). We count the power of 2 on both sides of equation (2). The power on the right hand side is at least the power of 2 in $(a_1+1)!$ which is $\lfloor \frac{a_1+1}{2}\rfloor + \lfloor \frac{a_1+1}{2^2}\rfloor +\cdots $. In the product of the left hand side, the power of 2 is at most $2(l_1+l_2-1)+\frac{\mathrm{log}(x_1+l_1)}{\mathrm{log}(2)}+\frac{\mathrm{log}(x_2+l_2)}{\mathrm{log}(2)}$. This is because the power of 2 in $\Delta(x,k)$ is at most $k-1+\frac{\mathrm{log}(x+k)}{\mathrm{log}(2)}$ (see \cite{NS}, p.320 for an argument). Thus
			$$
			\lfloor \frac{a_1+1}{2}\rfloor + \lfloor \frac{a_1+1}{2^2}\rfloor +\cdots \leq 2(l_1+l_2-1)+\frac{\mathrm{log}(x_1+l_1)}{\mathrm{log}(2)}+\frac{\mathrm{log}(x_2+l_2)}{\mathrm{log}(2)}. 
			$$
		But for the left hand side, we have
		$$
		\lfloor \frac{a_1+1}{2}\rfloor + \lfloor \frac{a_1+1}{2^2}\rfloor +\cdots >\frac{a_1+1}{2}+\cdots +\frac{a_1+1}{2^7}-7>0.99(a_1+1)-7.
		$$
		This implies
		$$
	0.99(a_1+1)\leq 2(l_1+l_2)+5+\frac{\mathrm{log}(x_1+l_1)}{\mathrm{log}(2)}+\frac{\mathrm{log}(x_2+l_2)}{\mathrm{log}(2)}.
		$$
		And since $x_1+l_1>x_2+l_2$, we get
		$$
			0.99(a_1+1)\leq 2(l_1+l_2)+5+2\frac{\mathrm{log}(x_1+l_1)}{\mathrm{log}(2)}.
		$$
		Now we assume $x_1>4l_1$. Since $l_1\geq l_2$, we have
		$$
			0.99(a_1+1)\leq 2(l_1+l_2)+5+2\frac{\mathrm{log}(x_1+l_1)}{\mathrm{log}(2)}<4l_1+5+2\frac{\mathrm{log}(2x_1)}{\mathrm{log}(2)}
		$$
		and therefore
		$$
		0.99a_1<4l_1+6.01+2\frac{\mathrm{log}(x_1)}{\mathrm{log}(2)}.
			$$
		Using Theorem 2.4, we conclude from (2) that
		$$
		a_1+1>4.42l_1
		$$
		if $(x_1,l_1)\notin T$. This yields
		$$
		0.99a_1<\frac{4}{4.42}a_1+6.915+2\frac{\mathrm{log}(a_1)}{\mathrm{log}(2)}, 
		$$
		exept for finitely many exceptions. This shows that $a_1$ is bounded when $(x_1,l_1)\notin T$. If $(x_1,l_1)\in T$, consider $x_1+l_1$ and notice that $x_1+l_1>x_2+l_2$. For these finitely many combinations of $x_1,l_1, x_2$ and $l_2$ there are obviously finitely many solutions to equation (2). Hence there are finitely many nontrivial solutions whenever $x_1>4l_1$.\\
		\noindent
		If $x_1\leq 4l_1$, the assertion follows directly from
		$$
			0.99(a_1+1)\leq 2(l_1+l_2)+5+2\frac{\mathrm{log}(x_1+l_1)}{\mathrm{log}(2)}<4l_1+5+2\frac{\mathrm{log}(5l_1)}{\mathrm{log}(2)}.
		$$ 
		This completes the proof. 
		
		\vspace{0.3cm}
		\noindent
		{\tiny HOCHSCHULE FRESENIUS UNIVERSITY OF APPLIED SCIENCES 40476 D\"USSELDORF, GERMANY.}\\
		E-mail adress: sasa.novakovic@hs-fresenius.de\\
		
	\end{document}